\documentclass[sn-mathphys,Numbered]{sn-jnl}

\usepackage[numbers,sort&compress]{natbib}
\usepackage{graphicx}%
\usepackage{multirow}%
\usepackage{amsmath,amssymb,amsfonts}%
\usepackage{amsthm}%
\usepackage{mathrsfs}%
\usepackage[title]{appendix}%
\usepackage{xcolor}%
\usepackage{textcomp}%
\usepackage{manyfoot}%
\usepackage{booktabs}%
\usepackage{algorithm}%
\usepackage{algorithmicx}%
\usepackage{algpseudocode}%
\usepackage{listings}%

\theoremstyle{thmstyleone}%
\newtheorem{theorem}{Theorem}[section]
\newtheorem{proposition}[theorem]{Proposition}%

\theoremstyle{thmstyletwo}%
\newtheorem{remark}{Remark}%
\newtheorem{lemma}[theorem]{Lemma}
\theoremstyle{thmstylethree}%
\newtheorem{question}{Question}
\newtheorem{problem}{Problem}
\newtheorem{cor}[theorem]{Corollary}
\begin{document}

\title[]{Cyclic Analytic $2$-isometry of Finite Rank and Cauchy Dual Subnormality Problem}


\author[1]{\fnm{} \sur{M. N.  Khasnis}}\email{mandar.khasnis@gmail.com}
\author[2]{\fnm{} \sur{V. M. Sholapurkar}}\email{vmshola@gmail.com}

\affil[1]{\orgdiv{Department of Mathematics}, \orgname{Smt. CHM College}, \orgaddress{\city{Ulhasnagar}, \postcode{421003}, \state{Maharashtra}, \country{India}}}

\affil[2]{\orgname{Bhaskaracharya Pratishthana}, \orgaddress{ \city{Pune}, \postcode{411004}, \state{Maharashtra}, \country{India}}}

\abstract{The Cauchy dual subnormality problem for $2$-isometries has attracted the attention of the researchers in recent years. In this article, we describe the problem and present a new counter-example to the problem by constructing a family of analytic, cyclic 2-isometric operators whose Cauchy dual is not subnormal. The said example is in contrast with the class of operators found in \cite{cgr2022} whose Cauchy dual is subnormal. In the process, we employ the technique of realizing a $2$-isometry as a shift operator on a suitable  de Branges-Rovnyak space. The construction of the counter example involves some unwieldy computations around the roots of a polynomial of degree four. We therefore use numerical techniques with the help of SageMath to facilitate the computational work. }




\maketitle

\section{Introduction and Preliminaries}\label{sec1}
The notion of Cauchy dual of an operator was introduced by Shimorin in \cite{Sh}.  The notion of Cauchy dual turns out to be interesting, especially while simultaneously studying two classes of  operators which are in some sense antithetical to each other. An excellent illustration of such classes is provided by contractive subnormal operators and completely hyperexpansive operators. The interplay between these two classes of operators can be thought of as an operator theoretic manifestation of two special classes of functions viz. completely monotone functions and completely alternating functions. Also, both the classes of operators have been described in terms of operator inequalities. This fruitful synthesis of harmonic analysis with operator theory has been  extensively explored in the works of Athavale and co-workers \cite{athavale_vms1999, athavale_ranjekar2002_1, athavale_ranjekar2002_2, athavale2003, sholapurkar2000} as well as in the works of Jablonski and Stochel \cite{jablonski2001, jablonski2002, jablonski2003, jablonski2004, jablonski2006}. It turns out that the notion of Cauchy dual fits naturally in these circumstances and helps in understanding the subtle interconnections.

The said theme begins with a simple observation :- If $T:\{\alpha_n\}$ is a weighted shift operator, then the Cauchy dual $T'$ is the weighted shift with weight sequence  $\Big\{\displaystyle\frac{1}{\alpha_n}\Big\}$. Combining this fact together with a result of \cite[Chapter 4, Proposition 6.10]{bcr1984} enabled Athavale \cite{At} to prove that the Cauchy dual of a completely hyperexpansive weighted shift is a contractive subnormal weighted shift. As a  generalization of this result, it is natural to ask the following question:
\\
\begin{question}
	Is the Cauchy dual of a completely hyperexpansive operator subnormal?\\
\end{question}
This question appears in \cite[Question 2.11]{Ch}. 
Though the question still remains open in this full generality, some special cases of this problem have been dealt with in recent years. Chavan et. al. have extensively contributed in solving the problem in some special cases.  The phrase `Cauchy dual subnormality problem' has evolved through these considerations. 

In order to set up the context, we make a brief mention of these special cases here, and elaborate on these results in the next section.


\begin{itemize}
	\item An affirmative answer to CDSP for a family of multiplication operators on Dirichlet-type spaces has been given in \cite{cgr2022}. The work relies on the results on operator theory in de Branges-Rovnyak spaces.
	\item First counterexample  to CDSP has  been found in \cite{sc2019}. A class of $2$-isometric weighted shifts on directed trees such that their Cauchy dual is not subnormal, has been constructed in this paper.
	\item A class of cyclic 2-isometric composition operators has been exhibited in \cite{sc2020} whose Cauchy dual are not subnormal.
\end{itemize}

  In this paper, capitalizing on the works of  \cite{costara2016} and \cite{cgr2022}, we produce a counter example to the CDSP by way of constructing an example of cyclic, analytic $2$-isometry on a Dirichlet-type space, whose Cauchy dual is not subnormal. The said Dirichlet-type space is obtained by choosing a finitely supported measure on the unit circle $\mathbb{T}$. Our construction relies on the fact that such a Dirichlet-type space coincides with a suitable de Branges-Rovnyak space. Moreover, the construction involves the computation of the roots of a polynomial of degree four. In this process, we employ some numerical techniques by using the SageMath tool. 
    
\subsection{Preliminaries}
Let $\mathbb{R}$ denote the set of real numbers, $\mathbb{C}$ denote the set of complex numbers, $\mathbb{D}$ denote the open unit disc and $\mathbb{T}$ denote the unit circle.

The notion of a reproducing kernel Hilbert space (RKHS) is well known and has been extensively studied, explored and crucially used in the theory of operators on Hilbert spaces. We quickly recall the definition of RKHS for a ready reference. \\ 


Let $S$ be a non-empty set and $\cal F$$(S,\mathbb{C})$ denote set of all functions from $S$ to $\mathbb{C}$. We say that $\cal H\subseteq \cal F$$(S,\mathbb{C})$ is a {\it \bf Reproducing Kernel Hilbert Space}, briefly {\it RKHS}, if
\begin{enumerate}
	\item[i.] $\cal H$ is a vector subspace of $\cal F$$(S,\mathbb{C})$
	\item[ii.] $\cal H$ is equipped with an inner product $\langle , \rangle$ with respect to which $\cal H$ is a Hilbert space
	\item[iii.] For each $x\in S$, the linear evaluation functional $E_x:\cal H$$\to \mathbb{C}$ defined by $E_x(f)=f(x)$, is bounded.
\end{enumerate}
If $\cal H$ is an RKHS on $S$ then for each $x\in S$, by Riesz representation theorem, there exists unique vector $k_x\in \cal H$ such that, for each $f\in \cal H$, $f(x)=E_x(f)=\langle f,k_x\rangle$.
\\The function $K:S\times S \to \mathbb{C}$ defined by $K(x,y)=k_y(x)$ is called as the reproducing kernel for $\cal H$. A classical reference for reproducing kernel Hilbert spaces is \cite{paulsen2016RKHS}.\\

We now recall the definitions of three special types of RKHS which are relevant in the present context. 
\begin{enumerate}
	
	\item {\bf Hardy Space:} Let Hol$(\mathbb{D})$ denote the set of all holomorphic functions on $\mathbb{D}$. The Hardy space $H^2$ is defined as
	\[H^2=\Bigg\{f(z)=\displaystyle\sum_{n=0}^{\infty}a_nz^n\in \text{Hol}(\mathbb{D}) : \displaystyle\sum_{n=0}^{\infty}|a_n|^2<\infty\Bigg\}.\]
	It is a well known fact that $H^2$ is a reproducing kernel Hilbert space with the inner product defined as $\langle f,g \rangle :=\displaystyle\sum_{n=0}^{\infty}a_n\overline{b_n}$, if $f(z)=\displaystyle\sum_{n=0}^{\infty}a_nz^n$ , $g(z)=\displaystyle\sum_{n=0}^{\infty}b_nz^n$ and a kernel(called as {\it Szego Kernel}) function $K_\lambda(z)=\displaystyle\frac{1}{1-\overline{\lambda}z}$. The interested reader may consult \cite{douglas2012}.\\
	
	\item {\bf Dirichlet-type spaces : }  For a finite positive Borel measure $\mu$ on unit circle $\mathbb{T}$, the Dirichlet-type space $D(\mu)$ is defined by
	\[D(\mu):=\Bigg\{f\in \text{Hol}(\mathbb{D}) : \displaystyle\int_{\mathbb{D}}|f'(z)|^2P_\mu(z)dA(z)<\infty \Bigg\}.\]
	Here, $P_\mu(z)$ is the Poisson integral for measure $\mu$ given by  $\displaystyle\int_{\mathbb{T}}\displaystyle\frac{1-|z|^2}{|z-\zeta|^2}d\mu(\zeta)$,\\ \noindent$dA$ denotes the normalised Lebesgue area measure on the open unit disc $\mathbb{D}$ and Hol$(\mathbb{D})$ denotes the set of all holomorphic functions on $\mathbb{D}$. 
	The space  $D(\mu)$ is a reproducing kernel Hilbert space and   multiplication by the coordinate function $z$ turns out to be a bounded linear operator (refer \cite{aleman1993}) on $D(\mu)$. For an elaborate discussion on Dirichlet type spaces the reader may also consult \cite{dirichlet_sp2014}. S. Richter described a model for a cyclic, analytic $2$-isometry\cite{richter1991}. Indeed, a  cyclic, analytic $2$-isometry is unitarily equivalent to  multiplication operator $M_z$ on a Dirichlet-type space $D(\mu)$ for some finite, positive Borel measure $\mu$ on unit circle $\mathbb{T}$.
	
	The study of Dirichlet type spaces by way of identifying such a space with a de Branges-Rovnyak space was initiated by D. Sarason. This association has been further strengthened in recent years in the works of  \cite{cgr2022, sarason97, cgr2010, kellay2015}. The identification of a Dirichlet space as a de Branges-Rovnyak space allows one to compute the reproducing kernel for the Dirichlet space. \\ Here, we include a brief description of de Branges-Rovnyak space for a ready reference. \\
	
	\item {\bf de Branges-Rovnyak spaces:}
	\\For complex, separable Hilbert spaces $\cal U$, $\cal V$, let $\cal B(U,V)$ denote the Banach space of all bounded linear transformations from $\cal U$ to $\cal V$. The {\bf Schur class} $S(\cal U,V)$ is given by
	\[S({\cal U,V})=\Big\{B:\mathbb{D} \to {\cal B(U,V)} : B \text{ is holomorphic, } \displaystyle\sup_{z\in \mathbb{D}}\|B(z)\|_{B({\cal U,V})}\leq 1\Big\}.\]
	Observe that when ${\cal U}=\mathbb{C}=\cal V$ then the Schur class is nothing but the closed unit ball of $H^\infty(\mathbb{D})$, the set of all bounded holomorphic functions on $\mathbb{D}$. For any $B\in S(\cal U,V)$, the de-Branges-Rovnyak space, $H(B)$ is the reproducing kernel Hilbert space associated with the $\cal B(V)$-valued semidefinite kernel given by
	\[K_B(z,w)=\displaystyle\frac{I_{\cal V}-B(z)B(w)^*}{1-z\overline{w}},~~z,w\in \mathbb{D}.\]
	For equivalent formulations of de Branges-Rovnyak spaces, the reader is referred to an excellent article by J. Ball \cite{Ball2015}. 
	The kernel $K_B$ is normalised if $K_B(z,0)=I_{\cal V}$ for every $z\in \mathbb{D}$. This is equivalent to the condition $B(0)=0$. Further, when ${\cal U}=\mathbb{C}=\cal V$, we denote $H(B)$ by $H(b)$, the classical de Branges-Rovnyak space (Refer \cite{fricain_mashreghi_2016vol2} for the basic theory of the classical de Branges-Rovnyak spaces). 
\end{enumerate}

Let $\cal H$ be a complex, infinite dimensional, separable Hilbert space and $\cal B(\cal H)$ denotes the $C^*$-algebra of bounded linear operators on $\cal H$. 
Here, we record a few definitions which are required in the sequel. 

Let $T\in {\cal B(H)}$.
We say that $T$ is {\it cyclic} if there exists a vector $f\in \cal H$ such that ${\cal H}=\vee \{T^nf : n\geq 0\}$(vector $f$ is known as {\it cyclic vector}). An operator $T$ is said to be {\it analytic} if $\cap_{n\geq 0}T^n{\cal H}=\{0\}$. The  {\it Cauchy dual} $T'$ of a left invertible operator $T\in \cal B(\cal H)$ is defined as $T'=T(T^*T)^{-1}$. Following Agler \cite{agler1995_1}, an operator $T\in \cal B(\cal H)$ is said to be a {\it $2$-isometry} if \[I-2T^*T+T^{*2}T^2=0.\]
An operator $T\in \cal B(H)$ is said to be {\bf subnormal } if there exist a Hilbert space $\cal K$ and an operator $S\in \cal B(K)$ such that $\cal H \subseteq K$, $S$ is normal and $S|_{\cal H}=T$. Readers are encouraged to refer \cite{conway1991} for the detailed study of subnormal operators. Agler proved in \cite{agler1985} that $T\in \cal B(H)$ is a subnormal contraction (means $T$ is subnormal and $\|T\|\leq 1$) if and only if 
\[B_n(T)\equiv \sum_{k=0}^{n}(-1)^k \binom{n}{k}T^{*k}T^k\geq 0 \text{ for all integers }n\geq 0.\]
Following \cite{At}, an operator $T\in B(\cal H)$ is said to be {\it completely hyperexpansive} if \[B_n(T)\leq 0~~ \text{for all integers }n\geq 1.\]

Now we are ready to state and describe the problem that has been tackled in this article. 
\subsection{The Problem} 
We begin our discussion with the fact \cite[Theorem 6.1]{cgr2022} which states that for a cyclic, analytic $2$-isometry $T$ in $\cal B(H)$, the rank of $T^*T-I$ is finite if and only if there exist a finitely supported measure $\mu$ on the unit circle $\mathbb{T}$  such that $T$ is unitarily equivalent to a multiplication by the coordinate function $z$ on $D(\mu)$. In fact, if the rank of $T^*T-I$ is a positive integer $k,$ then the corresponding measure $\mu$ is supported at  exactly $k$  points of unit circle $\mathbb{T}.$
In view of this result, a construction of a cyclic, analytic 2-isomtery reduces to choosing finitely many points on the unit circle and looking at the multiplication operator $M_z $ on the Dirichlet space $D(\mu)$, where $\mu$ is supported at the chosen points. Further, it is asserted that such a space $D(\mu)$  coincides with the \\ 
\noindent de Branges-Rovnyak space $H(B)$ for some $B$, with the equality of norms.

\noindent We now state the problem under consideration: 
 
\begin{problem}
	Characterize  finitely supported, positive, Borel measures $\mu$ on unit circle $\mathbb{T}$ such that the Cauchy dual of $M_z$ on $D(\mu)$ is subnormal.
\end{problem}

A solution to the problem in the  case when $\mu$ is supported at a single point is given  by the following result (see \cite[Corollary 3.6]{badea2019} and \cite[Corollary 5.4]{cgr2022}): Let $\lambda\in \mathbb{T}$ and $\gamma>0$. The Cauchy dual of $M_z$ on the Dirichlet-type space $D(\gamma \delta_\lambda)$ is a subnormal contraction.
Further, it is proved in \cite[Theorem 2.4]{cgr2022} that for a positive, Borel measure $\mu$ supported at any two antipodal points on unit circle $\mathbb{T}$, the Cauchy dual $M_z'$ of $M_z$ is subnormal. In view of these two results, it is natural to investigate the case of a measure supported at two non-antipodal points and address the problem in that case. 

Here, we consider the measures supported at two points of unit circle $\mathbb{T}$ such that the angle between them is $90^\circ$. Indeed, the main result of this paper is as given below.\\
\begin{theorem}\label{main}
	Let $\zeta_1$ and $\zeta_2$ be any two points on the unit circle $\mathbb{T}$ such that the angle between them is $90^\circ$(considered as vectors in $\mathbb{R}^2$). If $\mu=\delta_{\zeta_1}+\delta_{\zeta_2}$ is a measure on unit circle $\mathbb{T}$ then the Cauchy dual of $M_z$ on the Dirichlet-type space $D(\mu)$ is not subnormal.
\end{theorem}

\noindent The rest of the paper is devoted to the proof of this result. The nature of the proof is computational. Although, the computational details are not included in the proof, the same can be verified by an application of the SageMath code provided in the appendix.
 
\section{Proof of the main result}
As a special case, we choose points $\zeta_1=1$ and $\zeta_2=i$ and begin with the computation of reproducing kernel of the space $D(\mu)$, where $\mu$ is supported at $\{1,i\}$.

\subsection{Reproducing kernel of $D(\mu)$}\label{sec2}
As mentioned earlier, the space $D(\mu)$ is equipped with reproducing kernel. In this section, we take up the task of computing the reproducing kernel for the space $D(\mu)$ where $\mu$ is the regular, positive, Borel measure on the unit circle $\mathbb{T}$ supported on the points $\{1,i\}$. The technique used in obtaining the reproducing kernel heavily relies on the work of Costara \cite{costara2016}.  Thus, for a ready reference, we briefly outline the work carried out in \cite{costara2016} and state the results that are used in the sequel.\\
Given $n\in \mathbb{N}$. Let $c_1,c_2,\dots,c_n$ be strictly positive real numbers and $\zeta_1,\zeta_2,\dots,\zeta_n$ be pairwise distinct points on the unit circle $\mathbb{T}$. Consider the positive Borel measure $\mu=\sum_{j=1}^{n}c_j\delta_{\zeta_j}$ on $\mathbb{T}$, where $\delta_{\zeta_j}$ denotes the Dirac delta measure supported at $\zeta_j$.\\
As an application of Riesz-Fejer theorem, it has been observed in \cite{costara2016} that there exist $\alpha_1,\alpha_2,\dots,\alpha_n \in \mathbb{C}\setminus\overline{\mathbb{D}}$ and $d>0$ such that the following equation holds:
\begin{equation}\label{eq3}
	\prod_{j=1}^{n}|z-\zeta_j|^2+\sum_{j=1}^{n}c_j\prod_{\substack{i=1\\i\neq j}}^{n}|z-\zeta_i|^2=d\prod_{j=1}^n |z-\alpha_j|^2,~~~z\in \mathbb{T}
\end{equation}
Here we state the expressions for reproducing kernels of two subspaces of $H^2$ obtained in \cite[Theorem 3.1, 4.4]{costara2016}:
\begin{enumerate}
	\item The reproducing kernel for $O_\mu H^2$ is given by 
	\begin{equation}\label{eq4}
		\tilde{K_\mu}(z,\lambda)=\displaystyle \frac{O_\mu(z)}{1-\overline{\lambda}z}\overline{O_\mu(\lambda)}
	\end{equation}
	where $O_\mu(z)=\displaystyle\frac{p(z)}{q(z)}$, with $p(z)=\displaystyle\frac{e^{i\theta}}{\sqrt{d}}\displaystyle\prod_{j=1}^{n}(z-\zeta_j)$ ($\theta \in \mathbb{R}$ is chosen such that $O_\mu(0)>0$) and $q(z)=\displaystyle\prod_{j=1}^{n}(z-\alpha_j)$.
	\item The reproducing kernel for $(O_\mu H^2)^\perp$ is given by
	\begin{equation} \label{eq5}
		\hat{K_\mu}(z,\lambda)=\displaystyle\sum_{j=1}^{n}g_j(\lambda)f_j(z) ~~~~ z,\lambda \in \mathbb{D}
	\end{equation}\label{eq6}
	where \begin{equation}
		f_j(z)=\displaystyle\frac{O_\mu(z)}{O_\mu '(\zeta_j)(z-\zeta_j)}
	\end{equation}
	and 
	\begin{equation}\label{eq7}
		\begin{pmatrix}
			\overline{g_1(\lambda)}\\
			\overline{g_2(\lambda)}\\
			\vdots \\
			\overline{g_n(\lambda)}
		\end{pmatrix}=\left(\left[\langle f_i,f_j\rangle_\mu\right]_{1\leq i,j\leq n}\right)^{-1}
		\begin{pmatrix}
			f_1(\lambda)\\
			f_2(\lambda)\\
			\vdots \\
			f_n(\lambda)
		\end{pmatrix}
	\end{equation}	
	\item The Lemma \cite[Lemma 4.2, 4.3]{costara2016} provides formula for obtaining $\langle f_i,f_j\rangle_\mu$ for any $1\leq i,j\leq n$ as follows :
\begin{equation}\label{eq8}
		\|f_i\|_\mu^2=\langle f_i,f_i\rangle_\mu=c_i\zeta_i f_i'(\zeta_i)
	\end{equation} and for $i\neq j$,
	\begin{equation}\label{eq9}
		 \langle f_i,f_j\rangle_\mu= \displaystyle\frac{1}{O_\mu'(\zeta_i)\overline{O_\mu'(\zeta_j)}(1-\overline{\zeta_i}\zeta_j)}
\end{equation}
\end{enumerate}

\noindent The implementation of the above algorithm to the special case with $\zeta_1=1$, $\zeta_2=i$ yields the following theorem:\\

\begin{theorem}\label{thm1}
	Let $\mu$ be a positive Borel measure on unit circle $\mathbb{T}$ supported on $\{1,i\}$ defined as $\mu=\delta_1+\delta_i$. The reproducing kernel for $D(\mu)$ is given by:
	\begin{align*}
		K(z,\lambda)& =\displaystyle\frac{b}{\overline{q(\lambda)}q(z)}\Bigg[\displaystyle\frac{(\bar \lambda -1)(\bar\lambda +i)(z-1)(z-i)}{1-\bar\lambda z}\\ &+(a-1)\left((\bar\lambda +i)(z-i)+(\bar\lambda -1)(z-1)\right)\\ &+\displaystyle\frac{\bar s(\bar \lambda -1)(z-i)}{-i(w-pi)^2}+\frac{s(\bar \lambda +i)(z-1)}{i(w+pi)^2}\Bigg] ~~~~~~~(\lambda, z\in \mathbb{D})
	\end{align*} 
	where $a,b,p,w,s$ are constants and $q(z)$ is a polynomial of degree $2$ whose roots are in $\mathbb{C}\setminus \overline{\mathbb{D}}$.\\
\end{theorem}

\noindent Before we proceed with the proof of the theorem, we need to derive a special case of Riesz-Fejer theorem by choosing $n=2, \zeta_1=1, \zeta_2=i, c_1=c_2=1$ in (\ref{eq3}). Though, the existence of $\alpha_1, \alpha_2$ has been ensured by Riesz-Fejer theorem, we need the numerical values of $\alpha_1, \alpha_2$. These values turn out to be crucial in the proof of theorem \ref{thm1} and consequently for the main result (theorem \ref{main}) of the paper. 
 We first proceed with the following lemma, which gives us the values of $\alpha_1, \alpha_2$.\\
 
\begin{lemma}\label{l1}
	For $z\in \mathbb{T}\setminus\{1,i\}$, there exists $\alpha_1,\alpha_2\in \mathbb{C}\setminus \overline{\mathbb{D}}$ and hence a polynomial $q(z)$ of degree $2$ and a constant $d>0$ such that \[|z-1|^2|z-i|^2+|z-1|^2+|z-i|^2=d~ q(z)\overline{q(z)}\]
\end{lemma}
\begin{proof}
We have,
	\begin{align*}
		&|z-1|^2|z-i|^2+|z-1|^2+|z-i|^2\\
	&=(z-1)(\bar z-1)(z-i)(\bar z+i)+(z-1)(\bar z-1)+(z-i)(\bar z+i)\\
	&=8+3iz-3i\bar z-3z-3\bar z-iz^2+i\bar z^2\\
	&=8+3iz-\displaystyle\frac{3i}{z}-3z-\displaystyle\frac 3z-iz^2+\displaystyle\frac{i}{z^2}\\
	&=-\displaystyle\frac{i}{z^2}\left[z^4-(3+3i)z^3+8iz^2+(3-3i)z-1\right]
\end{align*}

\noindent Let $f(z)=z^4-(3+3i)z^3+8iz^2+(3-3i)z-1$. As we need numerical values of the roots of $f(z)$, the roots have been computed using SageMath. For computational details, readers may consult appendix (\ref{secA1}.\ref{a2}). Further, as expected, two of these roots lie outside the closed unit disc. We denote these two roots by $\alpha_1, \alpha_2$ and the other two roots by $\beta_1$ and $\beta_2$. Further, $\alpha_1, \alpha_2$ satisfy the relations:
\begin{equation}\label{eq10}
	\alpha_1+\alpha_2=a(1+i)~~~\text{ and }~~~\alpha_1\alpha_2=b~i
\end{equation}
for some $a,b\in \mathbb{R}$.\\
Now, define $q(z)$ as $q(z)=(z-\alpha_1)(z-\alpha_2)$. Then $q(z)$ can be written in the form 
\[q(z)=z^2-a(1+i)z+bi.\]
Further, it is easy to check that the other two roots $\beta_1, \beta_2$ of the polynomial $f(z)$ are associated with the polynomial $q(z)$ in the sense that 
\[|z-1|^2|z-i|^2+|z-1|^2+|z-i|^2=d~q(z)\overline{q(z)}\]
where $\beta_1\beta_2=d\times i$.

\end{proof}

\noindent Now, applying (\ref{eq4}), we define \begin{equation}\label{eq1}
	O_\mu(z)=\displaystyle\frac{(z-1)(z-i)}{\sqrt{d}q(z)}.
\end{equation}
Observe that we take $\theta=\frac \pi 2$ so that $O_\mu(0)>0$.
Thus, we obtain the reproducing kernel for $O_\mu H^2$ as follows: 	
	\begin{equation}\label{eq2}
		\tilde{K_\mu}(z,\lambda)=\displaystyle\frac{(\bar \lambda -1)(\bar \lambda +i)(z-1)(z-i)}{d~(1-\bar \lambda z)q(z)\overline{q(\lambda)}}~~~~~~ (\lambda, z\in \mathbb{D}).
	\end{equation}

\noindent Now, we implement (\ref{eq5}) to obtain the reproducing kernel of $(O_\mu H^2)^\perp$.\\

\begin{proposition}\label{p1}
	The reproducing kernel for $(O_\mu H^2)^\perp$ is given by 	
	\begin{align*}
		\hat{K_\mu}(z,\lambda)=\displaystyle\frac{b}{q(z)\overline{q(\lambda)}}&\Bigg[(a-1)\left((\overline{\lambda}+i)(z-i)+(\overline{\lambda}-1)(z-1)\right)\\ &+\frac{\overline{s}(\overline{\lambda}-1)(z-i)}{-i(w-pi)^2}+\frac{s(\overline{\lambda}+i)(z-1)}{i(w+pi)^2}\Bigg]~~~~(\lambda, z\in \mathbb{D})
	\end{align*}
	for some constants $p,w,s$ and a constant  $a=\displaystyle\frac{\alpha_1+\alpha_2}{1+i}$, as obtained in the equation (\ref{eq10}).
\end{proposition}
\begin{proof}
	We use equation (\ref{eq6}) to obtain the following expressions for $f_1(z)$ and $f_2(z)$ (using Sagemath) :\\ $f_1(z)=\displaystyle\frac{O_\mu(z)}{O_\mu'(1)(z-1)}=\displaystyle\frac{z-i}{\sqrt{d}~q(z)(-p-wi)}$ and \\
	$f_2(z)=\displaystyle\frac{O_\mu(z)}{O_\mu'(i)(z-i)}=\displaystyle\frac{z-1}{\sqrt{d}~q(z)(w+pi)}$ \\
	Where, $p,w$ are positive real constants (Refer Appendix \ref{secA1}.\ref{a2}).\\
	\noindent Now, equations (\ref{eq8}) and (\ref{eq9}) are used, to obtain the matrix: 
	 $B=\begin{bmatrix}
		\langle f_1,f_1\rangle & \langle f_1,f_2\rangle\\
		\langle f_2,f_1\rangle & \langle f_2,f_2\rangle
	\end{bmatrix}^{-1} $
	\\
	We use SageMath to compute the entries and observe that $B$ is of the form,\\
	 \[B=\begin{bmatrix}
		a-1 & s\\
		\bar s & a-1
	\end{bmatrix},\] where $s\in \mathbb{C}$ is a constant (See Appendix \ref{secA1}.\ref{a4}) and $`a$' as obtained in the proof of Lemma \ref{l1}.\\
\noindent Now we consider equation (\ref{eq7}) for the present context as follows : for any $\lambda \in \mathbb{D}$, \\
$\begin{bmatrix}
	g_1(\lambda)\\
	g_2(\lambda)
\end{bmatrix}=\overline{B} \times \begin{bmatrix}
	\overline{f_1(\lambda)}\\
	\overline{f_2(\lambda)}
\end{bmatrix}=\begin{bmatrix}
	a-1 & \bar s\\
	s & a-1
\end{bmatrix}
\begin{bmatrix}
	\displaystyle\frac{\overline{\lambda}+i}{\sqrt{d}~\overline{q(\lambda)}(-p+wi)}\\
	\displaystyle\frac{\overline{\lambda}-1}{\sqrt{d}~\overline{q(\lambda)}(w-pi)}
\end{bmatrix}$ (See Appendix \ref{secA1}.\ref{a4}).
\\This gives,\\ $g_1(\lambda)=\displaystyle\frac{1}{\sqrt{d}~\overline{q(\lambda)}}\left[\displaystyle\frac{(a-1)(\bar \lambda +i)}{-p+wi}+\displaystyle\frac{\bar s (\bar \lambda -1)}{w-pi}\right]$ and \\
$g_2(\lambda)=\displaystyle\frac{1}{\sqrt{d}~\overline{q(\lambda)}}\left[\frac{s(\bar \lambda +i)}{-p+wi}+\frac{(a-1)(\bar \lambda -1)}{w-pi}\right]$\\
Finally using equation (\ref{eq5}), we get the reproducing kernel of $(O_\mu H^2)^\perp$ as follows : 
	\begin{align*}
		\hat{K_\mu}(z,\lambda)
		&=g_1(\lambda)f_1(z)+g_2(\lambda)f_2(z)\\
		&=\displaystyle\frac{b}{q(z)\overline{q(\lambda)}}\Bigg[(a-1)\left((\overline{\lambda}+i)(z-i)+(\overline{\lambda}-1)(z-1)\right)\\ &+\frac{\overline{s}(\overline{\lambda}-1)(z-i)}{-i(w-pi)^2}+\frac{s(\overline{\lambda}+i)(z-1)}{i(w+pi)^2}\Bigg]~~~~(\lambda, z\in \mathbb{D}).
	\end{align*}
	This completes the proof.
\end{proof}
\noindent{\bf Proof of Theorem \ref{thm1} :}
\begin{proof}
	By \cite[Theorem 5.1]{costara2016}, the reproducing kernel of $D(\mu)$ is given as follows :\\

For $\lambda, z \in \mathbb{D}$,  
$K(z,\lambda)=\tilde{K_\mu}(z,\lambda)+\hat{K_\mu}(z,\lambda)$.\\

Thus, by equation (\ref{eq2}) and proposition \ref{p1} we get, 
\begin{align*}
	K(z,\lambda) & =\displaystyle\frac{b}{\overline{q(\lambda)}q(z)}\Bigg[\displaystyle\frac{(\bar \lambda -1)(\bar\lambda +i)(z-1)(z-i)}{1-\bar\lambda z}          \\
	             & +(a-1)\left((\bar\lambda +i)(z-i)+(\bar\lambda -1)(z-1)\right)                                                                                   \\
	             & +\displaystyle\frac{\bar s(\bar \lambda -1)(z-i)}{-i(w-pi)^2}+\frac{s(\bar \lambda +i)(z-1)}{i(w+pi)^2}\Bigg] ~~~~~~~(\lambda, z\in \mathbb{D}).
\end{align*}
\end{proof} 

\noindent The above computations for reproducing kernel of $D(\mu)$ further helps to discuss CDSP for $M_z$ on $D(\mu)$.

\subsection{Cauchy Dual Subnormality Problem}\label{sec3}
 As promised earlier, we now present a counter example to CDSP. We show that the Cauchy dual of $M_z$ on $D(\mu)$ is not subnormal, where $\mu$ is a finite, positive, Borel measure supported at $\{1,i\}$. Here we first record following two main results from \cite{cgr2022}.\\
 \begin{theorem}\cite[Theorem 6.4]{cgr2022}\label{thm2}
 	For positive scalars $c_1,c_2,\dots,c_k$ and distinct points $\zeta_1,\zeta_2,\dots,\zeta_k$ on the unit circle $\mathbb{T}$, consider the positive Borel measure $\mu=\sum_{j=1}^{k}c_j\delta_{\zeta_j}$ on $\mathbb{T}$, where $\delta_{\zeta_j}$ denotes the Dirac delta measure supported at $\zeta_j$. Let $X(z)=(z,z^2,\dots,z^k)^T$ and $\{e_j\}_{j=1}^k$ denote the standard basis of $\mathbb{C}^k$. Then there exist $\alpha_1, \alpha_2,\dots,\alpha_k\in \mathbb{C}\setminus\overline{\mathbb{D}}$ and a $k\times k$ upper triangular matrix $P$ such that the Dirichlet-type space $D(\mu)$ coincides with the de Branges-Rovnyak space $H(B)$ with equality of norms, where $B=\left(\frac{p_1}{q}, \frac{p_2}{q}, \dots, \frac{p_k}{q}\right)$ and 
 	\[p_j(z)=\left\langle PX(z),e_j\right\rangle~,~j=1,\dots,k,~~~q(z)=\prod_{j=1}^k (z-\alpha_j)\]
 	Moreover, $\alpha_1,\dots,\alpha_k$ are governed by
 	\[\prod_{j=1}^{k}|z-\zeta_j|^2+\sum_{j=1}^{k}c_j\prod_{\substack{i=1\\i\neq j}}^{k}|z-\zeta_i|^2=d\prod_{j=1}^k |z-\alpha_j|^2,~~~z\in \mathbb{T}\]
 	for some $d>0$.\\
 \end{theorem}

 \begin{theorem}\cite[Theorem 2.1]{cgr2022}\label{thm3}
 	Let $B=(b_1,\dots,b_k)\in {\cal S}(\mathbb{C}^k,\mathbb{C})$ be such that $B(0)=0$ where
 	\[b_j(z)=\displaystyle\frac{p_j(z)}{\prod_{j=1}^k (z-\alpha_j)}\]
 	for polynomials $p_j$ of degree at most $k$ and distinct numbers $\alpha_1,\dots,\alpha_k \in \mathbb{C}\setminus\overline{\mathbb{D}}$. for $r=1,\dots k$, let $a_r=\prod_{1\leq t\neq r \leq k} (\alpha_r-\alpha_t)$. Assume that the operator $M_z$ of multiplication by $z$ on the de Branges-Rovnyak space $H(B)$ is bounded. Then the Cauchy dual $M_z'$ of $M_z$ is subnormal if and only if the matrix 
 		\[\displaystyle\sum_{r,t=1}^{k}\left( \frac{1}{a_r\overline{a_t}}\sum_{j=1}^{k}p_j(\alpha_r)\overline{p_j(\alpha_t)}\right)\left(1-\frac{1}{\alpha_r \overline{\alpha_t}}\right)^l\left(\left(\frac{1}{\alpha_r^{m+2}\overline{\alpha_t}^{n+2}}\right)\right)_{m,n\geq 0}\] 
 		is formally positive semi-definite for every $l\geq 1$.
 \end{theorem}
 \begin{cor}\cite[Corollary 4.3]{cgr2022}\label{cor1}
 	Assume the hypothesis of Theorem \ref{thm3}. If $\alpha_r\overline{\alpha_t}\notin [1,\infty)$ for every $1\leq r\neq t\leq k$, then $M_z'$ is subnormal if and only if 
 	\[\displaystyle \sum_{j=1}^k p_j(\alpha_r)\overline{p_j(\alpha_t)}=0, ~~~1\leq r\neq t\leq k\]
 \end{cor}

\noindent We first find $B=(b_1,b_2)$ such that $D(\mu)$ coincides with $H(B)$ with equality of norms.\\
\begin{theorem}\label{thm4}
	Let $\mu$ be a positive Borel measure on unit circle $\mathbb{T}$ supported on $\{1,i\}$. Then the Dirichlet type space $D(\mu)$ coinsides with de Branges-Rovnyak space $H(B)$ with $B=(b_1,b_2)$ and $b_j=\displaystyle\frac{p_j}{q}$ where \[p_1(z)=p_{11}z+p_{12}z^2, ~p_2(z)=p_{22}z^2 ~~and ~~q(z)=(z-\alpha_1)(z-\alpha_2)\] with $\alpha_1, \alpha_2$ are as obtained in the Lemma \ref{l1} and $p_{11}, p_{12}$ and $p_{22}$ are constants.
\end{theorem}
\begin{proof}
	Let $\xi_1=1$, $\xi_2=i$, $B=(b_{ij})_{1\leq i,j\leq 2}$ be the matrix obtained in proposition \ref{p1}. Also, let $O_\mu(z)$ as defined by equation (\ref{eq1}).\\ 
	
	Since the multiplication operator $M_z$ on $D(\mu)$ is an analytic norm increasing operator (see \cite[theorem 3.6]{richter1991} and the reproducing kernel for $D(\mu)$ is normalized. By \cite[Lemma 3.4]{cgr2022}, there exists a positive semi-definite kernel $\eta:\mathbb{D}\times \mathbb{D}\to \mathbb{C}$ such that 
	\[\eta(z,0)=0,~~~k(z,w)=\displaystyle\frac{1-\eta(z,w)}{1-z\overline{w}},~~z,w\in \mathbb{D}\]
	Equating this with the reproducing kernel obtained in the equations (\ref{eq4}) and (\ref{eq5}), we get,
	\[q(z)\eta(z,w)\overline{q(w)}=q(z)\overline{q(w)}-p(z)\overline{p(w)}\left(1+(1-z\overline{w})\displaystyle\sum_{i,j=1}^{2}\frac{\overline{b_{ij}}}{O_\mu'(\xi_j)O_\mu'(\xi_i)}\frac{1}{(z-\xi_j)(\overline{w}-\overline{\xi_i})}\right)\]
	
	\noindent Since, the right hand side is a polynomial in $z$ and $\overline{w}$, there exists a matrix $\hat A=(a_{ij})_{0\leq i,j\leq 2}$ such that the above expression is equal to $\displaystyle\sum_{i,j=0}^{2}a_{ij}z^i\overline{w}^j$, $z,w\in \mathbb{D}$.\\
	But, this expression at $w=0$ is $0$ for each $z\in \mathbb{D}$, we obtain a matrix $A$ from $\hat A$ by deleting first row and first column of matrix $\hat A$.\\
	After performing all the computations(See Appendix \ref{secA1}.\ref{a6}), we get, \\
	\[a_{11}=17.13,~~a_{12}=-5.46 - 5.46i\]
	\[a_{21}=-5.46 + 5.46 i,~~a_{22}=4.72\]
\noindent Although, here throughout the discussion, the values are typed up to $2$ decimal places,  the actual work is carried out with at least $10$ digits after the decimal point.

\noindent Matrix $A$ is positive semi-definite and hence by applying Cholesky's decomposition (using Sagemath) we get an upper triangular $2\times 2$ matrix $P$ such that $A=P^*P$, where\\ \[P=\begin{bmatrix}
		4.14 & ~~~-1.32 - 1.32i\\
		0 & 1.11
	\end{bmatrix}.\]
\\Once again using Theorem \ref{thm2}, we get,
\begin{align*}
	p_1(z)=&\langle P (z,z^2)^T,(1,0)\rangle=4.14z+(-1.32 -1.32~i)z^2\\
	p_2(z)=&\langle P (z,z^2)^T,(0,1)\rangle=1.11 z^2
\end{align*}

\noindent Therefore, we conclude that the Dirichlet type space $D(\mu)$ coincides with de Branges-Rovnyak space $H(B)$ with $B=(b_1,b_2)$ and $b_j=\displaystyle\frac{p_j}{q}, j=1,2$. 
\end{proof}
These computations lead to a striking counter example of a finite rank analytic, cyclic $2$-isometry whose Cauchy dual is not subnormal. we consider $M_z$ on the above $D(\mu)$ and claim that its Cauchy dual $M_z'$ is not subnormal.\\ 

\noindent {\bf Proof of Theorem \ref{main}:}
\begin{proof}
	We now apply corollary \ref{cor1} to the polynomials $p_1$ and $p_2$ as obtained in theorem \ref{thm4}. The computations using Sagemath reveal that 
	\begin{enumerate}
		\item For $1\leq r\neq t\leq 2$,  $\displaystyle\sum_{j=1}^{2}p_j(\alpha_r)\overline{p_j(\alpha_t)}=-230.72\neq 0$ (see Appendix \ref{secA1}.\ref{a8})
		\item $\alpha_1\overline{\alpha_2}=0.97 - 5.38i\notin [1, \infty)$ and \\
		$\alpha_2\overline{\alpha_1}=0.97 + 5.38i \notin [1,\infty)$ (See Appendix \ref{secA1}.\ref{a7}).
	\end{enumerate}
	
	\noindent Now, it follows that the Cauchy dual $M_z'$ of the multiplication by $z$ operator $M_z$ on $D(\mu)$ where $\mu$ is a positive, Borel measure on the unit circle $\mathbb{T}$ given by $\mu=\delta_1+\delta_i$ is {\bf not subnormal}.\\
	In view of \cite[Proposition 7.1]{cgr2022}, we  conclude the proof of the theorem.
\end{proof}

\begin{remark}
	This provides a subclass of cyclic, analytic, 2-isometries of finite rank whose Cauchy dual is not subnormal.
\end{remark}


\section{Conclusion}\label{sec13}
The work carried out in this paper asserts that if a finitely supported measure $\mu$ on the unit circle $\mathbb{T}$ is of the form $\mu=\delta_{\zeta_1}+\delta_{\zeta_2}$ where $\zeta_1, \zeta_2\in \mathbb{T}$ are points such that the angle between them is $90^\circ$, then the Cauchy dual of the multiplication operator $M_z$ on the corresponding Dirichlet space $D(\mu)$ is not subnormal. It provides a class of operators which are cyclic, analytic, $2$-isometric of finite rank, such that their Cauchy dual is not subnormal. 
In this context, the following points deserve further exploration:

\begin{enumerate}
	\item If the measure $\mu$ is supported at points $\zeta_1$, $\zeta_2$ on the unit circle $\mathbb{T}$ having some other symmetry condition such as the angle between them is $120^\circ$, then is the Cauchy dual of $M_z$ on $D(\mu)$ subnormal?
	\item In view of a positive result in case of the measure supported at antipodal points and negative result in case of the measure supported at the points perpendicular to each other, it is interesting to find a symmetry condition on the support of the measure so that we have a complete solution to CDSP in this case.
	\item The authors have a feeling that the Cauchy dual is subnormal only in the case when the support of the measure is at antipodal points.
\end{enumerate}


\section*{Declarations}
\begin{itemize}
	\item Funding : Not applicable
	\item Conflict of interest/Competing interests (check journal-specific guidelines for which heading to use) : Not applicable
	\item Ethics approval : Not applicable
	\item Consent to participate : Not applicable
	\item Consent for publication : Yes
	\item Availability of data and materials : Not applicable
	\item Code availability : Yes
	\item Authors' contributions : Equal
\end{itemize}
The present work is carried out at the research center at the Department of Mathematics, S. P. College, Pune, India(autonomous).

\begin{appendices}

\section{Program codes in Sagemath}\label{secA1}
The following codes can be executed using CoCalc.
\begin{enumerate}
\item	\label{a1}
\begin{verbatim}
	z = PolynomialRing(ComplexField(), 'z').gen()
	pol=z^4-(3+3*I)*z^3+8*I*z^2+3*(1-I)*z-1
	pol.roots()
\end{verbatim}
\item \label{a2}
\begin{verbatim}
	a=2.53579711118167;b=5.46269136247034;c=0.464202888818329;d=0.183059948594236
	p=0.954692530486206;q=0.297593972106043
	alpha1=2.32798295504488 + 0.207814156136787*i;
	alpha2=0.207814156136787 + 2.32798295504488*i
\end{verbatim}
\item \label{a3}
\begin{verbatim}
	O(z)=((z-1)*(z-I))/(sqrt(d)*(z^2-a*(1+I)*z+b*I))
	Od(z)=derivative(O(z),z)
	Od(1)
	Od(i)
	f1(z)=(z-i)/(sqrt(d)*(z^2-a*(1+I)*z+b*I)*(-p-q*I))
	f1d(z)=derivative(f1(z),z)
	f2(z)=(z-1)/(sqrt(d)*(z^2-a*(1+I)*z+b*I)*(q+p*I))
	f2d(z)=derivative(f2(z),z)
\end{verbatim}
\item \label{a4} \begin{verbatim}
	1/((i-1)*(-p-q*I)^2)
	w=-0.695548570053500 - 0.127327085478790*I;
	wbar=-0.695548570053500 + 0.127327085478790*I
	A=matrix(2,2,[1.10401859575639,w,wbar,1.10401859575639])
	B=A.inverse()
	print(B)
	s=0.967575626606929 + 0.177124344467703*I
\end{verbatim}
\item \label{a5}
\begin{verbatim}
	qq(z)=(z-alpha1)*(z-alpha2)
	h1(z)=b*(1+I)/(qq(z)*conjugate(qq(1)))*((a-1)*(z-I)+(s*(z-1)/(I*(q+p*I)^2)))
	h1d(z)=derivative(h1(z),z)
	u1(z)=(h1(z)-h1(1))/(z-1)
	u2(z)=(h1(z)-h1(I))/(z-I)
	conjugate(h1d(0))+conjugate(u1(0))+conjugate(u2(0))
\end{verbatim}
\item \label{a6}
\begin{verbatim}
	m=s/((q+p*I)*(-p+q*I))
	n=conjugate(s)/((-p-q*I)*(q-p*I))
	
	a11=2*a^2+(-2-m-n-I*(m-n))/d
	a12=-a*(1+I)+(2+2*I-a*I-m-n*I-a)/d
	a21=-a*(1-I)+(2-2*I+a*I+m*I-n-a)/d
	a22=1+(2*a-3+m+n)/d
	
	D=matrix(2,2,[a11,a12,a21,a22])
	print(D)
	import numpy as np
	from scipy.linalg import cho_factor, cho_solve
	T=np.linalg.cholesky(D)
	
	P=matrix(2,2, [4.13925355+0.j,-1.31972862-1.31972862j,0,1.11084575+0.j])
	Q=matrix(2,2,[4.13925355,0,-1.31972862+1.31972862*I,1.11084575])
	print(P) 
	print(Q*P)
	
	p1(z)=4.13925355*z+(-1.31972862000000 -1.31972862000000*I)*z^2
		p2(z)=1.11084575*z^2
\end{verbatim}
\item \label{a7}
\begin{verbatim}
	a1=alpha1-alpha2;a2=alpha2-alpha1;
	k11=1/(a1*conjugate(a1))*(p1(alpha1)*conjugate(p1(alpha1))+
	p2(alpha1)*conjugate(p2(alpha1)))*(1-1/(alpha1*conjugate(alpha1)))^2
	k12=1/(a1*conjugate(a2))*(p1(alpha1)*conjugate(p1(alpha2))+
	p2(alpha1)*conjugate(p2(alpha2)))*(1-1/(alpha1*conjugate(alpha2)))^2
	k13=1/(a2*conjugate(a1))*(p1(alpha2)*conjugate(p1(alpha1))+
	p2(alpha2)*conjugate(p2(alpha1)))*(1-1/(alpha2*conjugate(alpha1)))^2
	k14=1/(a2*conjugate(a2))*(p1(alpha2)*conjugate(p1(alpha2))+
	p2(alpha2)*conjugate(p2(alpha2)))*(1-1/(alpha2*conjugate(alpha2)))^2
	m = matrix(CC, 1,1, lambda i, j:
	k11*(1/(alpha1^(i+2)*(conjugate(alpha1))^(j+2)))+
	k12*(1/(alpha1^(i+2)*(conjugate(alpha2))^(j+2)))+
	k13*(1/(alpha2^(i+2)*(conjugate(alpha1))^(j+2)))+
	k14*(1/(alpha2^(i+2)*(conjugate(alpha2))^(j+2)))); #print(m);
	det(m)
\end{verbatim}
\item \label{a8}
\begin{verbatim}
	alpha1*conjugate(alpha2)
	alpha2*conjugate(alpha1)
	p1(alpha1)*conjugate(p1(alpha2))+p1(alpha2)*conjugate(p1(alpha1))+
	p2(alpha1)*conjugate(p2(alpha2))+p2(alpha2)*conjugate(p2(alpha1))
\end{verbatim}
\end{enumerate}



\end{appendices}

\bibliographystyle{sn-basic}
\bibliography{ref.bib}

\end{document}